\documentclass[a4paper,11pt]{amsart}

\usepackage{latexsym}
\usepackage{amssymb}
\usepackage{amsthm}
\usepackage{amscd}
\usepackage{amsmath}
\usepackage{graphics}
\usepackage[all]{xy}

\newtheorem{Theorem}{Theorem}[section]
\newtheorem{Lemma}[Theorem]{Lemma}

\newtheorem{Proposition}[Theorem]{Proposition}
\newtheorem{Remark}[Theorem]{Remark}

\def\fka{{\frak a}}
\def\fkb{{\frak b}}

\def\fkm{{\frak m}}

\def\opn#1#2{\def#1{\operatorname{#2}}}

\opn\Spec{Spec}
\opn\Supp{Supp}
\opn\supp{supp}
\opn\Max{Max}
\opn\max{max}
\opn\Min{Min}
\opn\min{min}
\opn\Ass{Ass}
\opn\Assh{Assh}
\opn\Ann{Ann}
\opn\depth{depth}
\opn\rank{rank}

\opn\Mat{Mat}

\opn\Tot{Tot}

\opn\Sym{Sym}

\def\bsn{{\boldsymbol n}}

\opn\div{div}
\opn\Div{Div}
\opn\cl{cl}
\opn\Cl{Cl}

\opn\Ker{Ker}
\opn\Coker{Coker}
\opn\Im{Im}
\opn\Hom{Hom}
\opn\Tor{Tor}
\opn\Ext{Ext}
\opn\End{End}
\opn\Fitt{Fitt}
\opn\Aut{Aut}
\opn\id{id}

\opn\nat{nat}
\opn\pff{pf}
\opn\Pf{Pf}
\opn\GL{GL}
\opn\SL{SL}
\opn\G{G}
\opn\E{E}
\opn\H{H}
\opn\M{M}
\opn\mod{mod}
\opn\ord{ord}
\opn\det{det}
\opn\Soc{Soc}

\opn\chara{char}
\opn\length{\ell}
\opn\pd{pd}
\opn\rk{rk}
\opn\projdim{proj\,dim}
\opn\injdim{inj\,dim}
\opn\rank{rank}
\opn\depth{depth}
\opn\grade{grade}
\opn\height{ht}
\opn\embdim{emb\,dim}
\opn\codim{codim}

\renewcommand{\hat}{\widehat}

\title{A formula for the associated Buchsbaum-Rim multiplicities of a direct sum of cyclic modules} 
\author{Futoshi Hayasaka}
\address{Department of Environmental and Mathematical Sciences, 
Okayama University, 
3-1-1 Tsushimanaka, Kita-ku, Okayama, 700-8530, JAPAN}
\email{hayasaka@okayama-u.ac.jp}
\keywords{Buchsbaum-Rim multiplicity, Hilbert-Samuel multiplicity, cyclic modules}
\subjclass[2010]{Primary 13H15; Secondary 13P99}

\begin{document}

\maketitle

\begin{abstract}
In this article, we compute the Buchsbaum-Rim function of two variables associated to
a direct sum of cyclic modules and give a formula for the last positive associated Buchsbaum-Rim multiplicity 
in terms of the ordinary Hilbert-Samuel multiplicity of an ideal. 
This is a generalization of a formula for the last positive Buchsbaum-Rim multiplicity given 
by Kirby and Rees. 
\end{abstract}

\section{Introduction}


Let $(R, \fkm)$ be a Noetherian local ring with the maximal ideal $\fkm$ of dimension 
$d>0$ and let $C$ be a nonzero $R$-module of finite length. 
Let $\varphi: R^n \to R^r$ be an $R$-linear map of free modules with $C=\Coker \varphi$, 
and put $M:=\Im \varphi \subset F:=R^r$. Then 
one can consider the function  
$$\lambda_C(p):=\ell_R([\Coker \Sym_R(\varphi)]_{p})=\ell_R(S_{p}/M^{p}),  $$
where $S_p$ (resp. $M^p$) is a homogeneous component of 
degree $p$ of $S=\Sym_R(F)$ (resp. $R[M]=\Im \Sym_R(\varphi)$). The function of this type was introduced by 
Buchsbaum-Rim \cite{BR2} and they proved that $\lambda_C(p)$ is eventually a polynomial of degree $d+r-1$. 
Then they defined a multiplicity of $C$ as 
$$e(C):=\mbox{(The coefficient of} \ p^{d+r-1} \ \mbox{in the polynomial})\times (d+r-1)!, $$
which is now called the {\it Buchsbaum-Rim multiplicity} of $C$. They also proved that it
is independent of the choice of $\varphi$. 
Note that the Buchsbaum-Rim multiplicity $e(R/I)$ of a cyclic module $R/I$ defined 
by an $\fkm$-primary ideal $I$ in $R$ 
coincides with the ordinary Hilbert-Samuel multiplicity $e(I)$ of the ideal $I$. 

More recently, Kleiman-Thorup \cite{KT1, KT2} and Kirby-Rees \cite{KR1, KR2} introduced another kind of multiplicities 
which is related to the Buchsbaum-Rim multiplicity. They considered the function of two variables 
$$\Lambda(p, q):={\ell}_R(S_{p+q}/M^{p}S_{q}), $$
and proved that it is eventually a polynomial of total degree $d+r-1$. 
Then they defined a sequence of multiplicities of $C$ as, for $j=0, 1, \dots , d+r-1$, 
$$e^j(C):=(\mbox{The coefficient of} \ p^{d+r-1-j}q^j \ \mbox{in the polynomial})\times (d+r-1-j)!j! $$
and proved that it is independent of the choice of $\varphi$. Moreover they proved that  
$$e(C) = e^0(C) \geq e^1(C) \geq \dots \geq e^{r-1}(C)>e^r(C)= \dots = e^{d+r-1}(C)=0$$
where $r=\mu_R(C)$ is the minimal number of generators of $C$. Namely, 
the first multiplicity $e^0(C)$ is just the classical Buchsbaum-Rim multiplicity $e(C)$, and the sequence is always 
a descending sequence of non-negative integers with $e^j(C)=0$ if $j \geq r$, and $e^{r-1}(C)$ is the last positive
multiplicity. Thus the multiplicity $e^j(C)$ is now called {\it $j$-th Buchsbaum-Rim multiplicity} of $C$ or 
the {\it associated Buchsbaum-Rim multiplicity} of $C$. 

In this article, we investigate the detailed relation between the classical Buchsbaum-Rim multiplicity $e(C)=e^0(C)$ 
and the other one $e^j(C)$ for $j=1, 2, \dots , r-1$ by computing these invariants in a certain concrete case. 
There are some computation of the classical Buchsbaum-Rim multiplicity 
(see \cite{Bi, CLU, J, KR1, KR2} for instance). 
However, it seems that the computation of the other associated Buchsbaum-Rim multiplicities is done only 
for very special cases \cite{Ha, KR1, KR2}. 
One of the important cases is the case where $C=R/I_1 \oplus \dots \oplus R/I_r$ is a 
direct sum of cyclic modules. This case was first considered by Kirby-Rees \cite{KR1, KR2} and 
they gave an interesting formula for the classical Buchsbaum-Rim multiplicity $e(C)=e^0(C)$ in terms of mixed multiplicities of ideals (see also \cite{Bi} for more direct approach).

\begin{Theorem} {\rm (Kirby-Rees \cite{KR2})}
Let $I_1, \dots , I_r$ be $\fkm$-primary ideals in $R$. Then we have a formula 
$$e(R/I_1 \oplus \dots \oplus R/I_r)=\sum_{\stackrel{i_1, \dots , i_r \geq 0}{i_1+\dots +i_r=d}}e_{i_1 \cdots i_r}(I_1, \dots , I_r), $$
where $e_{i_1 \cdots i_r}(I_1, \dots , I_r)$ is the mixed multiplicity of $I_1, \dots , I_r$ of type $(i_1, \dots , i_r)$. 
\end{Theorem}

For the other multiplicities $e^j(R/I_1 \oplus \dots \oplus R/I_r)$ where $j=1, \dots , r-1$, 
Kirby-Rees \cite{KR2} considered the special case where $I_1 \subset \dots \subset I_r$ and proved the following. 

\begin{Theorem} {\rm (Kirby-Rees \cite{KR2})}
Let $I_1, \dots , I_r$ be $\fkm$-primary ideals in $R$. Suppose that $I_1 \subset \dots \subset I_r$. Then for any $j=1, \dots , r-1$, 
$$e^j(R/I_1 \oplus \dots \oplus R/I_r)=e(R/I_{j+1} \oplus \dots \oplus R/I_r).  $$
In particular, the last positive associated Buchsbaum-Rim
multiplicity 
$$e^{r-1}(R/I_1 \oplus \dots \oplus R/I_r)=e(R/I_r)$$
is the Hilbert-Samuel multiplicity of $I_r$.   
\end{Theorem}

The purpose of this article is to compute $e^j(R/I_1 \oplus \dots \oplus R/I_r)$ for any $\fkm$-primary 
ideals $I_1, \dots , I_r$ in $R$ and give a formula for the last positive associated Buchsbaum-Rim multiplicity $e^{r-1}(R/I_1 \oplus \dots \oplus R/I_r)$ in terms of the ordinary Hilbert-Samuel multiplicity of a sum of ideals.  
Here is the main result. 

\begin{Theorem}\label{main}
Let $I_1, \dots , I_r$ be arbitrary $\fkm$-primary ideals in $R$. Then we have a formula
$$e^{r-1}(R/I_1 \oplus \dots \oplus R/I_r)=e(R/I_1 + \dots + I_r).  $$
In particular, if $I_1, \dots , I_{r-1} \subset I_r$,  
$$e^{r-1}(R/I_1 \oplus \dots \oplus R/I_r)=e(R/I_r).$$
\end{Theorem}

This extends the Kirby-Rees formula 
for the last positive associated Buchsbaum-Rim multiplicity
and of our previous result \cite{Ha}. Our approach is a direct computation of the Buchsbaum-Rim function 
of two variables by using some ideas which is different from the one in \cite{KR2}. 
We note that it seems to be difficult to get the general formula by applying their approach \cite{KR2}. 
Moreover, our approach indicates the general formula for any other associated Buchsbaum-Rim multiplicities
$e^j(C)$ for $j=1, \dots , r-1$ which we will discuss and present it elsewhere. 

The proof of Theorem \ref{main} will be given in section 3.  
Section 2 is a preliminary character. 
In section 2, we will give a few elementary lemmas that we will use in the proof of Theorem \ref{main}. 
Our notation will be also fixed in this section. 

Throughout this article, let $(R, \fkm)$ be a Noetherian local ring with the maximal ideal $\fkm$ of dimension $d>0$. 
Let $r>0$ be a fixed positive integer and let $[r]=\{1, \dots , r\}$. For a finite set $A$, ${}^{\sharp} A$ denotes the number of elements of $A$. 
Vectors are always written in bold-faced letters, e.g., $\boldsymbol i =(i_1, \dots , i_r)$. 
We work in the usual multi-index notation. Let $I_1, \dots , I_r$ be ideals in $R$ and 
let $t_1, \dots , t_r$ be indeterminates. 
Then for a vector $\boldsymbol i =(i_1, \dots , i_r) \in \mathbb Z_{\geq 0}^r$, 
we denote $\boldsymbol I^{\boldsymbol i}=I_1^{i_1} \cdots I_r^{i_r}, \boldsymbol t^{\boldsymbol i}=t_1^{i_1} 
\cdots t_r^{i_r}$ and $| \boldsymbol i | =i_1+ \dots + i_r$. 
For vectors $\boldsymbol a, \boldsymbol b \in \mathbb Z^r$, 
$\boldsymbol a \geq \boldsymbol b \stackrel{{\rm def}}{\Leftrightarrow} a_i \geq b_i \ \mbox{for all} \ i=1, \dots , r.$
Let $\boldsymbol 0=(0, \dots , 0)$ be the zero vector in $\mathbb Z_{\geq 0}^r$ and let 
$\boldsymbol e=(1, 1, \dots , 1) \in \mathbb Z_{\geq 0}^{r}$.

\section{Preliminaries}

In what follows, let $I_1, \dots , I_r$ be $\fkm$-primary ideals in $R$ and let $C=R/I_1 \oplus \dots \oplus R/I_r$. 
In order to compute the associated Buchsbaum-Rim multiplicity of $C$, by taking a minimal free presentation $R^n \stackrel{\varphi}{\to} R^r \to C \to 0$
where the image of $\varphi$ is given by $M:=\Im \varphi=I_1 \oplus \dots \oplus I_r \subset F:=R^r$, we may  
assume that $S=R[t_1, \dots , t_r]$ is a polynomial ring and 
$R[M]=R[I_1t_1, \dots , I_rt_r]$ is the multi-Rees algebra of $I_1, \dots , I_r$. 
Then it is easy to see that for any $p, q \geq 0$, the module $M^pS_q$ can be expressed as 
$${\displaystyle M^{p}S_{q}=
\sum_{\substack{| \boldsymbol n | =p+q \\ \boldsymbol n \geq \boldsymbol 0} } 
\Bigg( \sum_
{\substack{| \boldsymbol i |=p \\ \boldsymbol 0 \leq \boldsymbol i \leq \boldsymbol n}} \boldsymbol 
I^{\boldsymbol i} 
\Bigg) 
\boldsymbol t^{\boldsymbol n}. }
$$
Here we consider a finite set $H_{p, q}:=\{ \boldsymbol n \in \mathbb Z_{\geq 0}^r \mid |\boldsymbol n |=p+q \}. $ 
For any $\boldsymbol n \in H_{p, q}$, let  
$${\displaystyle J_{p,q}({\boldsymbol n}):=\sum_
{\substack{| \boldsymbol i|=p \\ \boldsymbol 0 \leq \boldsymbol i \leq \boldsymbol n}} \boldsymbol I^
{\boldsymbol i}
}, $$
which is an ideal in $R$.  
Then the function $\Lambda(p,q)$ can be described as 
$${\displaystyle \Lambda(p, q) = \sum_{\boldsymbol n \in H_{p,q}} 
\ell_R(R/J_{p, q}({\boldsymbol n})). }$$
For a subset $\Delta \subset H_{p, q}$, we set 
$$\Lambda_{\Delta}(p, q):=\sum_{\boldsymbol n \in \Delta} 
\ell_R(R/J_{p, q}({\boldsymbol n})). $$
Here we define special subsets of $H_{p, q}$, which will be often used in the proof of Theorem \ref{main}. 
For $p, q>0$ and $k=1, \dots , r$, let 
$$\Delta_{p, q}^{(k)}:=\{\boldsymbol n \in H_{p, q} \mid n_1, \dots , n_k>p, n_{k+1}+ \dots + n_r \leq p \}. $$
With this notation, we begin with the following. 

\begin{Lemma}\label{lem1}
Let $p, q>0$ and $k=1, \dots , r$. 
Then for any $\boldsymbol n \in \Delta_{p, q}^{(k)}$, we have the equality 
$$J_{p, q}(\boldsymbol n)=(I_1+\dots +I_k)^{p-(n_{k+1}+\dots +n_r)} \prod_{j=k+1}^r(I_1+\dots +I_k+I_j)^{n_j}. $$
\end{Lemma}

\begin{proof}
Let $\boldsymbol n \in \Delta_{p, q}^{(k)}$. Then 
\begin{multline*}
J_{p, q}(\boldsymbol n) = \sum_{\substack{| \boldsymbol i|=p \\ \boldsymbol 0 \leq \boldsymbol i \leq \boldsymbol n}} \boldsymbol I^{\boldsymbol i}
\\
\shoveleft{=\sum_{\substack{0 \leq i_{k+1} \leq n_{k+1} \\ \cdots \\ 0 \leq i_r \leq n_r}} \Bigg( \sum_{\substack{i_1, \dots , i_k \geq 0 \\ i_1+\dots +i_k=p-(i_{k+1}+ \dots +i_r)}} I_1^{i_1} \cdots I_k^{i_k} \Bigg) I_{k+1}^{i_{k+1}} \cdots I_r^{i_r} 
}\\
\shoveleft{=\sum_{\substack{0 \leq i_{k+1} \leq n_{k+1} \\ \cdots \\ 0 \leq i_r \leq n_r}} (I_1+\dots +I_k)^{p-(i_{k+1}+\dots +i_r)} I_{k+1}^{i_{k+1}} \cdots I_r^{i_r} 
}\\
\shoveleft{= \sum_{\substack{0 \leq i_{k+1} \leq n_{k+1} \\ \cdots \\ 0 \leq i_r \leq n_r}} (I_1+\dots +I_k)^{p-(n_{k+1}+\dots +n_r)+(n_{k+1}-i_{k+1})+ \dots + (n_r-i_r)} I_{k+1}^{i_{k+1}} \cdots I_r^{i_r} 
}\\
\shoveleft{= (I_1+\cdots +I_k)^{p-(n_{k+1}+\dots +n_r)}\sum_{\substack{0 \leq i_{k+1} \leq n_{k+1} \\ \cdots \\ 0 \leq i_r \leq n_r}} (I_1+\dots +I_k)^{(n_{k+1}-i_{k+1})+ \dots + (n_r-i_r)} I_{k+1}^{i_{k+1}} \cdots I_r^{i_r}.} 
\end{multline*}

Here one can easily compute the above last sum as  
$$\sum_{\substack{0 \leq i_{k+1} \leq n_{k+1} \\ \cdots \\ 0 \leq i_r \leq n_r}} (I_1+\dots +I_k)^{(n_{k+1}-i_{k+1})+ \dots + (n_r-i_r)} I_{k+1}^{i_{k+1}} \cdots I_r^{i_r}
=\prod_{j=k+1}^r(I_1+\dots +I_k+I_j)^{n_j}. $$
Then we have the desired equality. 
\end{proof}

\begin{Lemma}\label{lem2}
Let $p, q>0$ with $q \geq (p+1)r$ and let $k=1, \dots , r$ and $0 \leq m \leq p$. Then  
$${}^\sharp \left\{ (n_1, \dots , n_k) \in \mathbb Z_{\geq 0}^k \left| \begin{array}{l}
n_1, \dots , n_k>p, \\
n_1+\dots +n_k=p+q-m
\end{array}
\right. 
\right\}={q-(k-1)p-1-m \choose k-1}. $$
\end{Lemma}

\begin{proof}
Let $S:=\{ (n_1, \dots , n_k) \in \mathbb Z_{\geq 0}^k \mid n_1, \dots , n_k>p, n_1+\dots +n_k=p+q-m \}$. 
Then the map
$\phi : S \to \{ (n_1, \dots , n_k) \in \mathbb Z_{\geq 0}^k \mid n_1+\dots +n_k=p+q-m-k(p+1) \}$
given by 
$\phi(\bsn)=\bsn-(p+1) \boldsymbol e$ is bijective so that the number ${}^\sharp S$
is just ${q-(k-1)p-1-m \choose k-1}$ which is the number of monomials of 
degree $p+q-m-k(p+1)$ in $k$ variables. 
\end{proof}

By Lemmas \ref{lem1} and \ref{lem2}, we have the explicit form of the function $\Lambda_{\Delta_{p, q}^{(k)}}(p, q)$. 

\begin{Proposition}\label{basicfunction}
Let $p, q>0$ with $q \geq (p+1)r$ and let $k=1, \dots , r$. Then 
$$\Lambda_{\Delta_{p, q}^{(k)}}(p, q)=\sum_{\stackrel{n_{k+1}, \dots , n_r \geq 0}{n_{k+1}+ \dots +n_r \leq p}} 
{q-(k-1)p-1-(n_{k+1}+\dots +n_r) \choose k-1} \ell_R(R/\fka),  $$
where $\displaystyle{
\fka:=(I_1+\dots +I_k)^{p-(n_{k+1}+\dots +n_r)} \prod_{j=k+1}^r(I_1+\dots +I_k+I_j)^{n_j}. 
}$

In particular, we have the inequality 
$$\Lambda_{\Delta_{p, q}^{(k)}}(p, q) \leq {q-(k-1)p-1 \choose k-1} \lambda_{L}(p), $$
where $\displaystyle{L=R/(I_1+\dots +I_k) \oplus \bigoplus_{j=k+1}^r R/(I_1+\dots +I_k+I_j)}$. 
\end{Proposition}

\begin{proof}

Let $p, q>0$ with $q \geq (p+1)r$ and let $k=1, \dots , r$. Then

\begin{multline*}
\Lambda_{\Delta_{p, q}^{(k)}}(p, q)=\sum_{\bsn \in \Delta_{p, q}^{(k)}} \ell_R(R/J_{p, q}(\bsn)) 
\\
\shoveleft{= \sum_{\bsn \in \Delta_{p, q}^{(k)}} \ell_R \bigg( 
R/(I_1+\dots +I_k)^{p-(n_{k+1}+\dots +n_r)} \prod_{j=k+1}^r(I_1+\dots +I_k+I_j)^{n_j}
\bigg) \ \ \ \mbox{by Lemma \ref{lem1}} 
}\\
\shoveleft{= \sum_{\substack{n_{k+1}, \dots , n_r \geq 0 \\ n_{k+1}+\dots +n_r \leq p}}  
\Bigg[
{}^\sharp \left\{ (n_1, \dots  , n_k) \in \mathbb Z_{\geq 0}^k 
\left| 
\begin{array}{l}
n_1, \dots , n_k>p, \\
n_1+\dots +n_k=p+q-(n_{k+1}+\dots +n_r)
\end{array}
\right. 
\right\} 
}\\ 
\times \  
\ell_R \bigg( R/(I_1+\dots +I_k)^{p-(n_{k+1}+\dots +n_r)} \prod_{j=k+1}^r(I_1+\dots +I_k+I_j)^{n_j} \bigg)  
\Bigg] 
\\
\shoveleft{= \sum_{\substack{n_{k+1}, \dots , n_r \geq 0 \\ n_{k+1}+\dots +n_r \leq p}} 
\Bigg[
{q-(k-1)p-1-(n_{k+1}+\dots +n_r) \choose k-1} 
}\\
\times \ \ell_R \bigg( R/(I_1+\dots +I_k)^{p-(n_{k+1}+\dots +n_r)} \prod_{j=k+1}^r(I_1+\dots +I_k+I_j)^{n_j} \bigg)  
\Bigg] \ \ \ \mbox{by Lemma \ref{lem2}}. 
\end{multline*}

This proves the first assertion. For the second one, we first note that the above last term is at most 
$$ {q-(k-1)p-1 \choose k-1} \sum_{\substack{n_{k+1}, \dots , n_r \geq 0 \\ n_{k+1}+\dots +n_r \leq p}}
\ell_R \big( R/(I_1+\dots +I_k)^{p-(n_{k+1}+\dots +n_r)} \prod_{j=k+1}^r(I_1+\dots +I_k+I_j)^{n_j} \big). $$
Then, since the above last sum is just the ordinary Buchsbaum-Rim function $\lambda_L(p)$ of 
$L:=R/(I_1+\dots +I_k) \oplus \bigoplus_{j=k+1}^r R/(I_1+\dots +I_k+I_j )$ by definition, we get the desired 
inequality. 
\end{proof}

\begin{Remark}\label{rem}
{\rm 
As stated in the above proof, the function $\lambda_L(p)$ in Proposition $\ref{basicfunction}$ is 
the ordinary Buchsbaum-Rim function of $L$ where $L$ is a direct sum of $(r-k+1)$ cyclic modules.  
Therefore the function $\lambda_L(p)$ is a polynomial function of degree $d+r-k$ for all large enough $p$. 
}
\end{Remark}

\section{Proof of Theorem \ref{main}}

We prove Theorem \ref{main}. We work under the same situation and use the same notation as in section 2. 
In order to investigate the asymptotic property of the function $\Lambda(p, q)$, 
we may assume that 
\begin{equation}\label{largepq}
q\geq(p+1)r \gg 0. 
\end{equation}
In what follows, we fix integers $p, q$ which satisfy the condition (\ref{largepq}). 
Let $H:=H_{p, q}$ and let $J(\boldsymbol n):=J_{p, q}(\boldsymbol n) $ for $\boldsymbol n \in H$. 
We note here that for any $\bsn \in H$, there exists $i=1, \dots , r$ such that $n_i > p$ 
because of the condition (\ref{largepq}).

Then the set $H$ can be divided by $r$-regions as follows: 
$$H=\coprod_{k=1}^rH^{(k)}, $$
where $H^{(k)}:=\{ \boldsymbol n \in H \mid {}^{\sharp} \{ i \mid n_i > p \}=k \}. $
Hence the function $\Lambda(p, q)$ can be expressed as follows: 
$$\Lambda(p, q)=\sum_{k=1}^r \Lambda_{H^{(k)}} (p, q). $$
Therefore it is enough to compute each function $\Lambda_{H^{(k)}}(p, q)$. 
When $k=r$, we can compute the function explicitly as follows. 

\begin{Proposition}\label{k=r}
$$\Lambda_{H^{(r)}} (p, q)={q-(r-1)p-1 \choose r-1} \ell_R (R/(I_1+\dots +I_r)^p). $$ 
\end{Proposition}

\begin{proof}
This follows from Proposition \ref{basicfunction} since $H^{(r)}=\Delta_{p, q}^{(r)}$. 
\end{proof}

Thus we can reduce the problem to compute functions $\Lambda_{H^{(k)}}(p, q)$ for $k=1, \dots , r-1$. 
Let $1 \leq k \leq r-1$. To compute $\Lambda_{H^{(k)}}(p, q)$, 
we divide $H^{(k)}$ into ${r \choose k}$-regions as follows: 
$$H^{(k)}=\coprod_{\stackrel{A \subset [r]}{{}^{\sharp}A=r-k}} D_A^{(k)}, $$
where $D_A^{(k)}:=\{ \bsn \in H^{(k)} \mid n_i >p \ \mbox{for} \ i \notin A, n_i \leq p \ \mbox{for} \ i \in A \}$. 
Then the function $\Lambda_{H^{(k)}}(p, q)$ can be expressed as follows: 
$$\Lambda_{H^{(k)}}(p, q)=\sum_{\stackrel{A \subset [r]}{{}^{\sharp}A=r-k}} \Lambda_{D_A^{(k)}} (p, q). $$

When $k=r-1$, we can also compute the function explicitly and get the inequality as in Proposition \ref{basicfunction}. 
Here is the inequality we will use in the proof of Theorem \ref{main}. 

\begin{Proposition}\label{k=r-1}
There exists a polynomial $g_{r-1}(X) \in \mathbb Q[X]$ of degree $d+1$ such that 
$$\Lambda_{H^{(r-1)}}(p, q) \leq {q-(r-2)p-1 \choose r-2}g_{r-1}(p). $$
\end{Proposition}

\begin{proof}
It is enough to show that for any $j=1, \dots , r$, 
$$\Lambda_{D_{\{j\}}^{(r-1)}}(p,q) \leq {q-(r-2)p-1 \choose r-2} \lambda_{L_j}(p), $$ 
where $L_j=R/ (I_1+ \dots +\hat{I_j}+ \dots +I_{r})\oplus R/(I_1+ \dots +I_{r})$ because the function 
$\lambda_{L_j}(p)$ is a polynomial function of degree $d+1$ (see Remark \ref{rem}). 
We may only consider the case where $j=r$. Then it follows directly from Proposition \ref{basicfunction} 
since $D_{\{r\}}^{(r-1)}=\Delta_{p, q}^{(r-1)}$. 
\end{proof}

When $1 \leq k \leq r-2$, we can get the same inequality, 
although the situation is not simple as in the case where $k=r-1$.

\begin{Proposition}\label{k<r-1}
For any $1 \leq k \leq r-2$, there exists a polynomial $g_{k}(X) \in \mathbb Q[X]$ of degree $d+r-k$ 
such that $$\Lambda_{H^{(k)}}(p, q) \leq {q-(k-1)p-1 \choose k-1}g_{k}(p). $$
\end{Proposition}

\begin{proof}
Let $1 \leq k \leq r-2$. To prove the desired inequality, it is enough to show that 
for any subset $A \subset [r]$ with ${}^{\sharp} A=r-k$, there exists a polynomial
$h_A(X) \in \mathbb Q[X]$ of degree $d+r-k$ such that   
$$\Lambda_{D_A^{(k)}}(p, q) \leq {q-(k-1)p-1 \choose k-1}h_{A}(p).  $$
 
To show this, we may only consider the case where $A=\{k+1, k+2, \dots , r\}$. 
We then put $D^{(k)}:=D^{(k)}_{\{k+1, \dots , r\}}$. 
To investigate $\Lambda_{D^{(k)}}(p, q)$, we divide $D^{(k)}$ into two-parts:  
$$D^{(k)}=E_{-}^{(k)} \coprod E_{+}^{(k)}, $$
where 
$$E_{-}^{(k)}:=\{ \bsn \in D^{(k)} \mid n_{k+1}+ \dots + n_r \leq p \}, $$
$$E_{+}^{(k)}:=\{ \bsn \in D^{(k)} \mid n_{k+1}+ \dots + n_r > p \}. $$

With this notation, we have the following two lemmas. 

\begin{Lemma}\label{k-}
Let $1 \leq k \leq r-2$. Then   
$$\Lambda_{E^{(k)}_{-}}(p, q) \leq {q-(k-1)p-1 \choose k-1} \lambda_L(p)$$
where $\displaystyle{L=R/(I_1+\dots +I_k) \oplus \bigoplus_{j=k+1}^r R/(I_1+\dots +I_k+I_j)}$.
\end{Lemma}

\begin{proof}
This follows from Proposition \ref{basicfunction} since $E_{-}^{(k)}=\Delta_{p, q}^{(k)}$. 
\end{proof}

\begin{Lemma}\label{k+}
Let $1 \leq k \leq r-2$. Then 
there exists a polynomial $h(X) \in \mathbb Q[X]$ of degree $d+r-k$ such that 
$$\Lambda_{E^{(k)}_{+}}(p, q) \leq {q-(k-1)p-1 \choose k-1} h(p). $$ 
\end{Lemma}

\begin{proof}
Let $1 \leq k \leq r-2$. Then we first note that for any $\bsn \in E_{+}^{(k)}, $
\begin{eqnarray}
J(\bsn)&=&\sum_{\substack{\boldsymbol 0 \leq \boldsymbol i \leq \bsn \\ | \boldsymbol i |=p}} \boldsymbol I^{\boldsymbol i} \notag \\
&=&\sum_{\substack{0 \leq i_{k+1} \leq n_{k+1} \\ \dots \\ 0 \leq i_r \leq n_r \\ i_{k+1}+\dots +i_r \leq p}}
\Bigg( \sum_{\substack{i_1, \dots , i_k \geq 0 \\ i_1+ \cdots +i_k=p-(i_{k+1}+ \cdots + i_r) }} 
I_1^{i_1} \cdots I_k^{i_k} \Bigg) I_{k+1}^{i_{k+1}} \cdots I_r^{i_r}  \notag \\
&=& \sum_{\substack{0 \leq i_{k+1} \leq n_{k+1} \\ \dots \\ 0 \leq i_r \leq n_r \\ i_{k+1}+\dots +i_r \leq p}} 
(I_1+ \dots +I_k)^{p-(i_{k+1}+ \dots + i_r)} I_{k+1}^{i_{k+1}} \cdots I_r^{i_r}. 
\end{eqnarray}

Here we claim the following.

\medskip

{\bf Claim 1.} There exists an $\fkm$-primary ideal $\fkb$ in $R$ such that 
for any $\bsn \in E_{+}^{(k)},$  
$$\ell_R(R/J(\bsn)) \leq \ell_R(R/\fkb^p). $$ 

\medskip

Let $\fkb$ be an $\fkm$-primary ideal in $R$ such that $\fkb \subset I_j$ 
for any $j=k+1, \dots , r$ (such as $\fkb=I_{k+1} \cdots I_r$).
Let $\bsn \in E_{+}^{(k)}. $ Since $n_{k+1}, \dots , n_r \geq 0$ and $n_{k+1}+\dots +n_r > p$, 
there exist integers $a_{k+1}, \dots , a_r \in \mathbb Z$ 
such that 
$$\left\{ 
\begin{array}{l}
0 \leq a_j \leq n_j  \ \mbox{for any} \ j=k+1, \dots ,  r,  \ \mbox{and}\\
a_{k+1}+ \cdots +a_r=p. 
\end{array}
\right.  $$
Then, by the above expression (2) of $J(\bsn)$, 
$J(\bsn)
\supset I_{k+1}^{a_{k+1}} \cdots I_r^{a_r} 
\supset  \fkb^{a_{k+1}+\cdots + a_r} 
= \fkb^p$. 
Hence we have $\ell_R(R/J(\bsn)) \leq \ell_R(R/\fkb^p)$. 

\medskip

Therefore
$$\Lambda_{E_{+}^{(k)}} (p, q)
= \sum_{\bsn \in E_{+}^{(k)}} \ell_R(R/J(\bsn))
\leq \sum_{\bsn \in E_{+}^{(k)}} \ell_R(R/\fkb^p) 
={}^{\sharp} E_{+}^{(k)} \cdot \ell_R(R/\fkb^p). 
$$

{\bf Claim 2. } There exists a polynomial $h^{\circ}(X) \in \mathbb Q[X]$ of degree $r-k$ such that 
$${}^{\sharp} E_{+}^{(k)} \leq {q-(k-1)p-1 \choose k-1} \cdot h^{\circ}(p). $$

To show this, we divide $E_{+}^{(k)}$ as follows:
\begin{eqnarray*}
E_{+}^{(k)} &=& \{ \bsn \in H^{(k)} \mid n_1, \dots , n_k > p, n_{k+1}, \dots , n_r \leq p, n_{k+1}+ \cdots + n_r > p \} \\
&=& \coprod_{\substack{0 \leq n_{k+1} \leq p \\ \cdots \\ 0 \leq n_r \leq p \\ n_{k+1}+ \cdots + n_r > p}} F(n_{k+1}, \dots , n_r)
\end{eqnarray*}
where  
$$F(n_{k+1}, \dots , n_r):=\left\{(n_1, \dots , n_r) \in \mathbb Z_{\geq 0}^r  \left| 
\begin{array}{l}
n_1, \dots , n_k > p, \\
n_1+ \cdots + n_k = p+q-(n_{k+1}+ \dots + n_r) 
\end{array}
\right.
\right\}. $$
Therefore 
\begin{eqnarray*}
{}^{\sharp}E_{+}^{(k)}&=&\sum_{\substack{0 \leq n_{k+1} \leq p \\ \cdots \\ 0 \leq n_r \leq p \\ n_{k+1}+ \cdots + n_r > p}} {}^{\sharp} F(n_{k+1}, \dots , n_r) \\
&=& \sum_{\substack{0 \leq n_{k+1} \leq p \\ \cdots \\ 0 \leq n_r \leq p \\ n_{k+1}+ \cdots + n_r > p}}
{q-(k-1)p-1-(n_{k+1}+ \dots + n_r) \choose k-1} \ \ \ \ \ \mbox{by Lemma \ref{lem2}} \\
&\leq & {q-(k-1)p-1 \choose k-1} \cdot
{}^{\sharp} \left\{ (n_{k+1}, \dots, n_r) \in \mathbb Z_{\geq 0}^{r-k} \left| 
\begin{array}{l}
n_{k+1}, \dots , n_r \leq p, \\
n_{k+1}+ \cdots +n_r > p 
\end{array}
\right. 
\right\} \\
&\leq & {q-(k-1)p-1 \choose k-1} \cdot 
{}^{\sharp} \{ (n_{k+1}, \dots, n_r) \in \mathbb Z_{\geq 0}^{r-k} \mid p < n_{k+1}+ \cdots +n_r \leq (r-k)p \} \\
&=& {q-(k-1)p-1 \choose k-1} \cdot \left\{ {r-k+(r-k)p-1 \choose r-k}-{r-k+p-1 \choose r-k} \right\}. 
\end{eqnarray*}
This proves Claim 2. 

\medskip

Consequently, we have that 
$$\Lambda_{E_{+}^{(k)}} (p, q) 
\leq {}^{\sharp} E_{+}^{(k)} \cdot \ell_R(R/\fkb^p) 
\leq {q-(k-1)p-1 \choose k-1} \cdot h^{\circ}(p) \cdot \ell_R(R/\fkb^p).$$ 
Since $\ell_R(R/\fkb^p)$ is a polynomial in $p$ of degree $d$, 
the polynomial $h(X)$ which corresponds to $h(p)= h^{\circ}(p) \cdot \ell_R(R/\fkb^p)$ is our desired one. 
\end{proof}

By Lemmas \ref{k-} and \ref{k+}, 
$$\Lambda_{D^{(k)}}(p,q)
=\Lambda_{E_{-}^{(k)}}(p,q)+\Lambda_{E_{+}^{(k)}}(p,q) 
\leq {q-(k-1)p-1 \choose k-1}\Big(\lambda_L(p)+h(p)\Big). 
$$
This proves Proposition \ref{k<r-1}. \end{proof}

\medskip

Now let me give a proof of Theorem \ref{main}.

\begin{proof}[Proof of Theorem \ref{main}]
By Propositions \ref{k=r-1} and \ref{k<r-1}, for any $k=1, \dots , r-1$, 
there exists a polynomial $g_k(X) \in \mathbb Q[X]$ of degree $d+r-k$ such that 
$$\Lambda_{H^{(k)}}(p, q) \leq {q-(k-1)p-1 \choose k-1}g_k(p). $$
Since $\displaystyle{\Lambda(p, q)=\Lambda_{H^{(r)}}(p, q)+\sum_{k=1}^{r-1} \Lambda_{H^{(k)}}(p, q)}, $
we have that by Proposition \ref{k=r}, 
\begin{eqnarray*}
\Lambda(p, q)-{q-(r-1)p-1 \choose r-1} \ell_R(R/(I_1+ \dots +I_r)^p) &\leq & \sum_{k=1}^{r-1} {q-(k-1)p-1 \choose k-1} g_k(p). 
\end{eqnarray*}
Therefore, there exists a polynomial $g(X, Y) \in \mathbb Q[X, Y] $ with $\deg_Y g(X, Y) \leq r-2$ such that
$$\Lambda(p, q)-{q-(r-1)p-1 \choose r-1} \ell_R(R/(I_1+ \dots +I_r)^p) \leq g(p, q). $$
The LHS in the above inequality is a polynomial function of two variables with non-negative integer values so that
the function $\Lambda(p, q)$ can be expressed as  
$$\Lambda(p, q)={q-(r-1)p-1 \choose r-1} \ell_R(R/(I_1+ \dots +I_r)^p)+f(p, q) $$
for some $f(X, Y) \in \mathbb Q[X, Y]$ with $\deg_Y f(X, Y) \leq r-2$. 
Then, by comparing the coefficients of $p^dq^{r-1}$ in the above equality, we obtain the equality
$$e^{r-1}(R/I_1 \oplus \dots \oplus R/I_r)=e(R/I_1+ \dots +I_r). $$
Then we get the desired formula. 
\end{proof}

\section*{Acknowledgments}
The author would like to thank the referee for his/her careful reading and constructive suggestions.



\end{document}